\documentclass[article]{amsart}
\usepackage[english]{babel}
\usepackage[utf8]{inputenc}
\usepackage{amsmath}
\usepackage{graphicx}
\usepackage[colorinlistoftodos]{todonotes}
\usepackage{amssymb}
\usepackage{amscd} 		
\usepackage{dsfont}
\usepackage{fixltx2e}
\usepackage{hyperref} 
\usepackage{cite}
\usepackage{tikz-cd}
\usepackage{listings}

\usepackage[margin=1in,top=0.75in,bottom=1in]{geometry} 

\newtheorem{theo}{Theorem}[section]
\newtheorem*{theo*}{Theorem}
\newtheorem{lemm}{Lemma}

\newtheorem{conj}{Conjecture}
\newtheorem*{conj*}{Conjecture}
\newtheorem{coro}{Corollary}
\newtheorem{obse}{Observation}
\newtheorem{prop}[theo]{Proposition}
\newtheorem{rema}{Remark}
\theoremstyle{definition} \newtheorem{exam}{Example} \newtheorem*{defi}{Definition}

\newcommand{\CC}{\mathbb{C}}

\newcommand{\PP}{\mathbb{P}}

\newcommand{\ZZ}{\mathbb{Z}}

\newcommand{\cO}{\mathcal{O}}
\newcommand{\cL}{\mathcal{L}}

\newcommand{\cT}{\mathcal{T}}

\newcommand{\cJ}{\mathcal{J}}

\newcommand{\rk}{\operatorname{rk}}

\def\and{\quad\mathrm{and}\quad}

\def\beq{\begin{equation}}
\def\eeq{\end{equation}}

\newcommand{\id}{\mathds{1}}

\title[Characteristic Classes of E.I.D.V]{Characteristic Classes of Homogeneous Essential Isolated Determinantal Varieties}
 \author{Xiping Zhang}

\date{\today}

\begin{document}
\begin{abstract}
The (homogeneous) Essentially Isolated Determinantal Variety is the natural generalization of generic determinantal variety, and is fundamental example to study non-isolated singularities. 
In this paper we study the characteristic classes on these varieties. We give explicit formulas of their Chern-Schwartz-MacPherson classes and Chern-Mather classes via standard Schubert calculus. As corollaries we obtain formulas for their (generic) sectional Euler characteristics, characteristic cycles and polar classes. 
%
\end{abstract}
\maketitle
\section{Introduction}
The study of characteristic classes and geometric invariants on singular spaces has been a major task in singularity theory and algebraic geometry, and have been intensely studied for the last decades. 
In the smooth setting, the Euler characteristic of a space is the degree of its total Chern class via Poincar\'e-Hopf theorem. For singular varieties the existence of such singular Chern classes was conjectured by Deligne-Grothendieck, and was proved by MacPherson in \cite{MAC} over $\CC$. 
Another  definition of singular Chern classes was due to M.-H. Schwartz, who used obstruction theory and radial frames to construct such classes \cite{MR32:1727}~\cite{MR35:3707}~\cite{MR629125}. In~\cite{MR629125} it was shown that these classes correspond, by Alexander isomorphism, to the classes defined by MacPherson. This cohomology class is  called the Chern-Schwartz-MacPherson class, denoted by $c_{sm}^X$ for any variety $X$. The integration of $c_{sm}^X$ equals the Euler characteristic of $X$.

The two important ingredients MacPherson used to define Chern-Schwartz-MacPherson class class are the local Euler obstruction and the Chern-Mather class, denoted by $Eu_X$ and $c_M^X$ respectively. They were originally defined on $\CC$ via topological method, and later in \cite{Gonzalez} Gonz\'{a}lez-Sprinberg  proved an equivalent algebraic intersection formula. His formula extends the definitions to arbitrary algebraically closed base field. Based on such algebraic formula, later in~\cite{Kennedy} G. Kennedy generalized the theory of Chern-Schwartz-MacPherson class to arbitrary algebraically closed field of characteristic $0$, using Sabbah's Lagrangian intersections and Chow groups. Thus in this paper we will work algebraically with algebraically closed field of characteristic $0$, and we will
view $c_{sm}^X$ as a class in the Chow group.

The goal of this paper is to explicitly compute the mentioned characteristic classes of
homogeneous essentially isolated determinantal varieties. These varieties are natural generalizations of the generic determinantal varieties, and are fundamental examples of non-isolated singularities.
Let $K$ be an algebraically closed field of characteristic $0$. Let $M_n$, $M_n^S$ and $M_n^\wedge$ be the spaces of $n\times n$ ordinary, symmetric and skew-symmetric matrices respectively. We will denote them by $M^*_n$, for $*$ denotes $\emptyset$, $S$ and $\wedge$ respectively. 
 They have natural stratifications $M^*_n=\cup_i \Sigma^*_{n,i}$, where the strata are matrices of fixed corank $i$. We consider transverse maps $F\colon V=K^N\to M_n^*$, where transverse means the image of $F$ intersects the non-zero strata $\Sigma^*_{n,i}$ transversely. By homogeneous we mean that $F$ is a $K^*$-equivariant map, where $K^*$ acts on $V$ and $M_n^*$ by scalar multiplications. Then the pull back orbits of $\Sigma^*_{n,i}$ are necessarily cones, and we call their projectivizations \textit{homogeneous Essentially Isolated Determinantal Varieties}. We will denote them by EIDV in short. For details we refer to \cite{G-E09}.

In \S\ref{S; preliminary} we review the theory of characteristic classes for (quasi) projective varieties. We briefly recall the definitions and basic properties of the Chern-Schwartz-MacPherson class and the Chern-Mather class. For projective varieties these are polynomials with variable $H=c_1(\cO(1))$. 
Then we recall the  involution  proposed by Aluffi in \cite{Aluffi13}. This involution $\cJ$ translate the information of Chern-Schwartz-MacPherson class of a projective variety $X$ to the information of the Euler characteristics of $X\cap L^k$ for  generic codimension $k$ linear subspaces. This reduces the computation of sectional Euler characteristics to the computation of Chern-Schwartz-MacPherson classes. 

The main formulas for the Chern classes of EIDV are presented in  \S\ref{S; charclass}. First we show that, via transversal pull-back of Segre-MacPherson  classes proved in \cite{Ohmoto}, it's enough to compute the Chern classes of generic determinantal varieties. For such varieties we
use their canonical resolutions: the Tjurina transforms. We define the $q$ polynomials to be the pushforward of the Chern-Schwartz-MacPherson classes of the Tjurina transforms, and show that
the Chern-Schwartz-MacPherson classes and the Chern-Mather classes of determinantal varieties are linear combinations of the $q$ polynomials.  
   
Our first formula is Theorem~\ref{theo; formulaI}, which interprets the coefficients of the $q$ polynomials  by integrations of tautological classes over Grassmannians. Here by tautological we mean the
Chern classes of the universal sub and quotient bundles. Thus our formula is purely combinatorial and can be easily computed by Macaulay2.
We present some computed examples in Appendix \S\ref{S; Appendix}.

Based on the fact that the function values at integers uniquely determine a polynomial, we also propose another equivalent formula (Theorem~\ref{theo; formulaII}).  For each type (ordinary, symmetric or skew-symmetric) of matrix we define the  determinantal Chow (cohomology) classes $Q_{n,r}$, $Q^\wedge_{n,r}$ and $Q^S_{n,r}$. 
These are Chow(cohomology) classes expressed in terms of  the tautological (sub or quotient)
bundles on  Grassmannians. Then we show  that the $q$ polynomials equal the integrations of these determinantal classes with the total Chern classes over the Grassmannians.
\begin{theo*}
Let $S$ and $Q$ be the universal sub and quotient bundles over the Grassmannian $G(r,n)$. 
We define the ordinary, symmetric ans skew-symmetric determinantal classes as follows.
\begin{tiny}
\begin{align*}
& Q_{n,r}(d)
:=  \left(\sum_{k=0}^{n(n-r)} (1+d)^{n(n-r)-k}c_k(Q^{\vee n})  \right) \left(\sum_{k=0}^{nr} d^{nr-k}c_k(S^{\vee n}) \right); \\
& Q^\wedge_{n,r}(d): =\\
& \left(\sum_{k=0}^{\binom{n-r}{2}} (1+d)^{\binom{n-r}{2}-k}c_k(\wedge^2 Q^\vee)  \right) 
\left(\sum_{k=0}^{\binom{r}{2}} d^{\binom{r}{2}-k} c_k(\wedge^2 S^\vee) \right)  
\left(\sum_{k=0}^{r(n-r)} d^{r(n-r)-k} c_k(S^\vee\otimes Q^\vee) \right) ; \\
& Q^S_{n,r}(d) :=\\
&\left( \sum_{k=0}^{\binom{n-r+1}{2}}   (1+d)^{\binom{n-r+1}{2}-k}c_k(Sym^2 Q^\vee)  \right)
 \left(\sum_{k=0}^{\binom{r+1}{2}} d^{\binom{r+1}{2}-k}  c_k(Sym^2 S^\vee) \right)  
  \left(\sum_{k=0}^{r(n-r)} d^{r(n-r)-k} c_k(S^\vee\otimes Q^\vee) \right) \/.
\end{align*} 
\end{tiny}
We have the following integration formulas : 
\begin{small}
\begin{align*}
 q_{n,r}(d)
=& \int_{G(r,n)} c(S^\vee\otimes Q)\cdot Q_{n,r}(d)\cap [G(r,n)] - d^{n^2} \binom{n}{r};  \\
 q^\wedge_{n,r}(d)
=&  
\int_{G(r,n)} c(S^\vee\otimes Q)  \cdot Q^\wedge_{n,r}(d)\cap [G(r,n)]  
- d^{\binom{n}{2}} \binom{n}{r} \/; \\
  q^S_{n,r}(d) \/;
=& 
\int_{G(r,n)} c(S^\vee\otimes Q)  \cdot Q^S_{n,r}(d) \cap [G(r,n)]
-d^{\binom{n+1}{2}} \binom{n}{r} \/. 
\end{align*}
\end{small}
\end{theo*} 
Notice that the (affine cones of) generic skew-symmetric and symmetric determinantal varieties are orbits of the $GL_n(\CC)$ representations. When the base field is $\CC$, in \cite{FR18} \cite{PR19} the authors used the method of axiomatic interpolation to
compute the $GL_n(\CC)$ equivariant Chern-Schwartz-MacPherson classes for the degeneracy loci. The equivariant Chern-Schwartz-MacPherson classes can be expressed as   polynomials in the weights of the  action,   or polynomials in the symmetric functions in formal Chern roots. The coefficients are not only numbers given by
complicated integrals, but some standard symmetric functions.
Since the actions contain $\CC^*$ scalar multiplication, one can specialize  the $GL_n(\CC)$ equivariant  Chern-Schwartz-MacPherson classes to the $\CC^*$ equivariant  Chern-Schwartz-MacPherson classes by identifying all the Chern roots to $t$. It was shown in \cite{Weber12} that, for any projective variety $X\subset \PP(V)$, the $\CC^*$ equivariant  Chern-Schwartz-MacPherson class  of the affine cone $\Sigma\subset V$ equals the ordinary Chern-Schwartz-MacPherson class of $X$, by changing $t$ to $H$. Thus despite the very different lookings, all the  formulas in  \cite{FR18} and \cite{PR19} evaluate  to the ones  in  this paper. It should be interesting to explain this fact combinatorically.

For complex projective varieties the theory of Chern-Schwartz-MacPherson classes are the pushdown of the theory of (Lagrangian) characteristic and conormal cycles.   
In \S\ref{S; charcycle} we briefly review the story and apply our results  to obtain formulas for the characteristic cycle classes of EIDV, as Chow classes of $\PP^{N} \times \PP^{N}$ (Proposition~\ref{prop; charcycle}). 
For generic determinantal varieties, bases on the local Euler obstruction formula proved in \cite{PR19} and \cite{Xiping3} we also compute their conormal cycle classes. 
Since the coefficients of the conormal cycle class are the degrees of the polar classes, we also obtain explicit formulas for the polar degrees of generic determinantal varieties (Equation~\ref{eq; polar}). For EIDV the conormal cycles depend on the local Euler obstruction information, thus combining with \S\ref{S; charclass} we obtain an algorithm to compute the polar degrees of EIDV.
We finish this section by proving an interesting observation: the characteristic cycles  of the closed orbits of all singular matrices are symmetric (Proposition~\ref{prop; charcyclesymmetry}). Such symmetry deserves a geometric explanation.

The Appendix \S\ref{S; Appendix} is devoted to explicit examples. 
The computations in the examples are based on
our formulas and  are carried out with the software Macaulay2~\cite{M2}. 
We highlight the patterns proved in the previous sections by the examples. We observe
that all the nonzero coefficients 
appearing in the Chern classes are positive. Moreover, all the polynomials and sequences presented in the examples are log concave.
These facts call for a conceptual, geometric explanation. 
Thus we close this paper with the  non-negative conjecture and the log concave conjecture  (Cf \S\ref{conj}).  
The situation appears to have similarities with the case of Schubert varieties in flag manifolds, which was recently proved in \cite{AMSS17}.
\section*{Acknowledgment}
The author is grateful to Paolo Aluffi,  Richard Rim\'anyi, Terrence Gaffney and Matthias Zach for many useful discussions and suggestions.  The author would like to thank Xia Liao for carefully reading the first draft.
The author also would like to thank the referees for all the comments. The author is supported by China Postdoctoral Science Foundation (Grant No.2019M661329).

\section{Preliminary}
\label{S; preliminary}
\subsection{Chern-Schwartz-MacPherson Class}
Let $X\subset \PP^N$ be a projective variety. The group of constructible function is defined as the abelian group generated by indicator functions $\id_V$ for all irreducible subvarieties $V\subset X$. We define the pushforward for a proper morphism $f\colon X\to Y$ as follows. For any closed subvariety $V\subset X$, the pushforward $F(f)(\id_V)(y)$ evaluates $\chi(f^{-1}(y)\cap V)$ for any $y\in Y$. 
This makes $F$ a functor from projective complex varieties  to  the abelian group category. 

The group $F(X)$ has $\{\id_V| V \text{ is a closed subvariety of } X \}$ as a natural base. 
In 1974  MacPherson defined a local  measurement for  singularities and 
names it the local Euler obstruction. He
proved that the local Euler obstruction functions $\{Eu_V|\text{V is a subvariety of X}\}$ also form a base for $F(X)$.
Based on this property he defined a natural transformation $c_*\colon F(X)\to H_*(X)$ that sends the local Euler obstruction function $Eu_V$ to Mather's Chern class $c_M^{V}$. He then proved the following theorem
\begin{theo}[\cite{MAC}]
The natural transformation $c_*$  is the unique natural transformation from $F$ to the homology functor $H_*$ satisfying the following normalization property: $c_*(\id_X)=c(TX)\cap [X]$ when $X$ is smooth.
\end{theo}
In 1990 Kennedy modified Sabbah's Lagrangian intersections and proved the following generalization.
\begin{theo}[\cite{Kennedy}]
Replace the homology functor by the Chow functor, 
MacPherson's natural transform extends to arbitrary algebraically closed field of characteristic $0$.
\end{theo}
%

Recall that the Chow group(ring) of $\PP^N$ is $\ZZ[H]/H^{N+1}$, where $H$ here is the hyperplane class $c_1(\cO(1))\cap [\PP^N]$.
We make the following definitions
\begin{defi}
Let $X\subset \PP^N$ be a projective subvariety. The Chern-Schwartz-MacPherson class and the Chern-Mather class of $X$, denoted by $c_{sm}^X(H)$ and $c_M^X(H)$, are defined as the pushforward of $c_*(\id_X)$ and $c_*(Eu_X)$ in $A_*(\PP^N)$.
\end{defi}
Notice that when $X$ is smooth, the Chern-Mather class and the Chern-Schwartz-MacPherson class all equal   the total Chern class $i_*(c(TX)\cap [X])$.

\begin{rema}
Let $X\subset \PP^N$ be a projective variety with  constant map $k\colon X\to \{p\}$. Then for any subvariety $Y\subset X$, the covariance property of $c_*$
shows that
\begin{align*}
\int_X c_{sm}^Y
=&~ \int_{\{p\}} Afc_*(\id_Y)=\int_{\{p\}} c_*Ff(\id_Y)\\
=&~ \int_{\{p\}} \chi(Y)c_*(\id_{\{p\}})=\chi(Y). \qedhere
\end{align*}
This observation gives a generalization of the classical
Poincar\'e-Hopf Theorem to possibly singular varieties.
\end{rema}

The theory of characteristic classes can also be generalized to motivic settings. For definitions, properties and examples we refer to \cite{BSY10}. In \cite{FRW18} the authors propose an axiomatic approach for such classes; recently in \cite{AC18} the authors applied such theory on pointed Brill-Noether problems. In this paper we only consider ordinary characteristic classes.

\subsection{Chern  Classes and Sectional Euler Characteristics}
In this subsection we introduce involutions defined by Aluffi in  \cite{Aluffi13} that connects the Chern-Schwartz-MacPherson class and sectional Euler characteristics.
Let $f(x)\in \ZZ[x]$ be a polynomial.  
We define   
$
\cJ \colon \ZZ [x]\to \ZZ [x]
$ by setting
\[
\cJ  \colon  f(x)\mapsto   \frac{xf(-1-x)-f(0)}{1+x} \/.  
\]

\begin{prop} 
One can observe the following properties for  $\cJ$ by direct computations:
\begin{enumerate}
\item  For any polynomial $f$ with no constant term,
$\cJ(\cJ(f))=f$. Thus $\cJ$ is an involution  on the set of polynomials with no constant term.
\item The involutions  $\cJ$ is linear, i.e., 
$\cJ (af+bg)=a\cJ (f)+b\cJ(g)$. 
\end{enumerate}
\end{prop}

Let $X\subset \PP(V)$ be a projective variety of dimension $n$.  For any $r\geq 0$ we define 
\[
X_r=X\cap H_1\cap \cdots \cap H_r
\]
to be the intersection of $X$ with $r$ generic hyperplanes. Let $\chi(X_r)=\int_{X_r} c_{sm}(X_r)$ be its Euler characteristic, we define $\chi_X(t)=\sum_i \chi(X_r)\cdot (-t)^r$ to be the corresponding \textit{sectional Euler characteristic polynomial}. On the other hand, 
write $c_{sm}^X=\sum_{i\geq 0}   \gamma_{N-i} H^i$  we define the $\gamma$ polynomial $\gamma_X(t):=\sum_i \gamma_i t^i $ by switching the variable from $H^i$ to $[\PP^i]$.
The polynomials $\chi_X(t)$ and $\gamma_X(t)$ are  polynomials of degree $\leq n$. 
\begin{theo}[\cite{Aluffi13}]
\label{theo; involution}
The involution $\cJ$ interchanges $\gamma_X(t)$ and $\chi_X(t)$:
$$
\cJ(\gamma_X(t))=\chi_X(t); \quad \cJ(\chi_X(t))=\gamma_X(t) \/.
$$
\end{theo}
This theorem shows that, the coefficients appeared in the Chern-Schwartz-MacPherson class of $X\subset \PP(V)$ are equivalent to the sectional Euler characteristics $\chi(X\cap L_r)$. Thus we can use the Chern-Schwartz-MacPherson class  to study the linear sections.

\subsection{Essentially Isolated Determinantal Varieties}
The Essentially Isolated Determinantal Singularities (EIDS) was introduced in \cite{G-E09}, as a generalization of determinantal type singularities. Let $K$ be a characteristic $0$ algebraically closed field. 
Let $M_n$, $M^S_n$ and $M^\wedge_n$ be the space of $n\times n$ ordinary, symmetric and skew-symmetric matrices over $K$ respectively. When the matrix type is not specified, we use $*$ to denote the upper-script.
We consider  maps $F=(f_{i,j})_{n\times n}\colon K^{N+1}\to M^*_n$  that intersect transversely along all the non-zero rank strata $\Sigma^{*\circ}_{n,k}$ of $M^*_n$. Here $\Sigma^{*\circ}_{n,k}$ denotes the stratum consisting matrices of rank $n-i$. The map $F$ may not be transversal to the origin in $M^*$. 
However, in this paper we always assume that $F$ is homogeneous, i.e.,  $f_{i,j}'s$ are homogeneous polynomials of degree $d$.  We consider the projectivization map $F\colon \PP(K^{N+1})\to \PP(M^*_n)$. Let $\tau^{* \circ}_{n,i}$ be the projectivization of $\Sigma^{*\circ}_{n,k}$, and let $\tau^*_{n,i}$ be its closure.
We define $X^*_{n,i}:=F^{-1}(\tau^*_{n,i})\subset \PP^{N}$ as the preimage of $\tau^*_{n,i}$. We call these varieties the \textit{Essentially Isolated Determinantal varieties}, and throughout this paper we  will use EIDV in short. We call the varieties $\tau_{n,i}^*$ \textit{generic determinantal varieties}.
\begin{prop} The following properties follow naturally from affine to projective setting.
\begin{enumerate}
\item The map $F$ intersect transversely to the strata $\tau_{n,i}^\circ$. 
\item Let $X_{n,i}^{* \circ}$ be the preimage of $\tau_{n,i}^{* \circ}$ for $i\geq k$,  then they form   a stratification of $X^*_{n,k}$. 
\item $X^*_{n,k}$ is smooth on the open  stratum $X^{* \circ}_{n,k}$.
The singularities of the closure $X^{*}_{n,k}$ are contained in $X^{*}_{n,k+1}$. 
\item The tautological line bundle of $\PP(M^*_n)$ pulls back to the $d$-tensor tautological line bundle of $\PP^{N}$, i.e., $F^*(\cO_{\PP(M^*_n)}(1))=\cO_{\PP^{N}}(d)$.
\end{enumerate}
\end{prop}
For detailed definitions and more properties  we refer to \cite{G-E09}\cite{NG-TG}.
\begin{exam}
The following two maps
\[
F\colon \CC^4\to M_{2,3};\quad 
\begin{bmatrix}
x_1 \\
\cdots \\
x_4
\end{bmatrix} 
\mapsto
\begin{bmatrix}
x_3 & x_2+x_4 & x_1\\
x_4 & x_1 & x_2 
\end{bmatrix} 
;
G\colon \CC^4\to M_{2,3}:
\begin{bmatrix}
x_1 \\
\cdots \\
x_4
\end{bmatrix} 
\mapsto
\begin{bmatrix}
x_1 & x_2 & x_3\\
x_2 & x_3  & x_4 
\end{bmatrix} 
\]
are both EIDS of degree $1$. The following map
\[
P\colon \CC^5\to M_{2,3};\quad 
\begin{bmatrix}
x_1 \\
\cdots \\
x_5
\end{bmatrix} 
\mapsto
\begin{bmatrix}
x_1^2+x_2^2 & x_2x_1 & x_3^2+x_4^2\\
x_4x_3 & x_3^2+x_4^2 & x_5^2 
\end{bmatrix} 
\]
is an EIDS of degree $2$.
\end{exam}

\section{Characteristic Class of EIDV}
\label{S; charclass}
In this section we compute the Chern-Scwartz-MacPherson classes of the EIDV. 
First we show that it's enough to compute the Chern classes for generic determinantal varieties. 

\begin{theo}[\textbf{Reduction to Generic Rank Loci}] 
\label{theo; csm}
For  $*$ substituted by $\emptyset$, $S$ and $\wedge$, which correspond  to ordinary, skew-symmetric and symmetric cases, we have the following formulas:
\[
c_{sm}^{X^*_{n,k}}(H)
= \frac{(1+dH)^{\dim M^*_n}}{(1+H)^{N+1}}  \cdot c_{sm}^{\tau^*_{n,k}}(dH) \/.
\]
\end{theo}
\begin{proof}
We consider the pullback of characteristic classes from determinantal varieties to EIDV. 
As shown in \cite{Yokura02}, the Chern-Schwartz-MacPherson classes don't behave very well under pull back, i.e., Verdier-Riemann-Roch for Chern-Schwartz-MacPherson classes fails in general. However, under our transversality assumption on $F$ the Verdier-Riemann-Roch holds for our case. 
This is due to the pullback property of the Segre-MacPherson class defined by T. Ohmoto in \cite{Ohmoto}, which we now recall.
For any  closed embedding $X\to M$ into smooth ambient space, the Segre-MacPherson class of $X$ is defined as 
\[
s^{SM}(X, M):=Dual(c(TM)^{-1}\cap c_*(X))\in A^*(M)\/.
\] 
Here $Dual$ denotes the Poincare dual of the ambient space $A^*(M)\sim A_*(M)$. 
Let $f\colon M\to N$ be a morphism of Whitney stratified smooth compact complex varieties, and let $Y$ be a  closed subvariety  of $N$.
Assume that $f$ intersects transversely with any strata of $Y$.  Ohmoto in \cite{Ohmoto} proved that
\[
f^*(s^{SM}(Y, N))=s^{SM}(f^{-1}(Y), M) \/.
\]
Since we require transversality in the definition of EIDV, we then have
\begin{align*}
c_{sm}^{X^*_{n,k}} \in A_*(\PP^N)
=&\frac{c(F^*\cO_{\PP(M^*_n)}(1))^{\dim M^*_n}}{c(\cO(1))^{N}} \cap F^* c_{sm}^{\tau^*_{n,k}}\\
=& \frac{(1+dH)^{\dim M^*_n}}{(1+H)^N} \cdot c_{sm}^{\tau^*_{n,k}}(dH)
\end{align*}
\end{proof}

This shows that the computation of  the Chern classes of EIDV is equivalent to the computation of Chern classes of determinantal varieties, for which we have the following. 
\begin{theo}[\textbf{Main Formula I}]
\label{theo; formulaI}
Denote $S$ and $Q$ to be the universal sub and quotient bundle over the Grassmanian $G(k,n)$. For $k\geq 1$, $i,p=0,1\cdots e^*$, we define the following Schubert integrations:
\begin{align*}
A_{i,p}(n,k) &:=\int_{G(k,n)} c(S^\vee\otimes Q)c_i(Q^{\vee n})c_{p-i}(S^{\vee n}) \cap [G(k,n)] \\
A^S_{l,i,p}(n,k) &:=\int_{G(k,n)} c(S^\vee\otimes Q)c_i(Sym^2 Q^\vee)
s_{\frac{k(2n-k+1)}{2}-l+p-i}(Sym^2 Q^\vee) \cap [G(k,n)] \\
A^\wedge_{l,i,p}(n,k) &:=\int_{G(k,n)} c(S^\vee\otimes Q)c_i(\wedge^2 Q^\vee)
s_{\frac{k(2n-k-1)}{2}-l+p-i}(\wedge^2 Q^\vee) 
\cap [G(k,n)] \/;
\end{align*}
and  the following binomials:
\[
B_{i,p}(n,k)  :=\binom{n(n-k)-p}{i-p};\quad 
B^S_{i,p}(n,k)  :=\binom{\binom{n-k+1}{2}-p }{ i-p};\quad
B^\wedge_{i,p}(n,k)  :=\binom{\binom{n-k}{2}-p }{ i-p}  \/.
\]
Here $e=n(n-r)$, $e^S=\binom{n-r+1}{2}$ and $e^\wedge=\binom{n-r}{2}$ correspond to the ranks of the vector bundles. Let $H$ be the hyperplane class in $\PP(M^*_n)$, we define the following $q$ polynomials for $k\geq 1$:
\begin{align*}
q_{n,k} & :=\sum_{l=0}^{n^2-1} 
\left( \sum_{p=0}^{n(n-r)}\sum_{i=0}^{p}  A^S_{l,i,p}(n,k)\cdot B^S_{p,i}(n,k) \right) H^{l} ; \\
q^S_{n,k}& :=
\sum_{l=0}^{\binom{n+1}{2}-1} 
\left( \sum_{p=0}^{\binom{n-k+1}{2}}\sum_{i=0}^{p}  A^S_{l,i,p}(n,k)\cdot B^S_{p,i}(n,k) \right) H^{l} ; \\
q^\wedge_{n,k} & :=
\sum_{l=0}^{\binom{n}{2}-1}
\left(
\sum_{p=0}^{\binom{n-k}{2}}\sum_{i=0}^{p}  A^\wedge_{l,i,p}(n,k)   \cdot  B^\wedge_{p,i}(n,k) 
\right) H^l \/.
\end{align*}
For ordinary rank loci, when  $k\geq 1$ we have:
\begin{equation}
c_M^{\tau_{n,k}}=  q_{n,k}; \quad 
c_{sm}^{\tau_{n,k}^{\circ}}
= \sum_{r=k}^{n-1} (-1)^{r-k}\binom{r}{k} \cdot q_{n,r} \/.
\end{equation}
For symmetric rank loci, when $k\geq 1$  we have
\begin{equation}
c_{sm}^{\tau_{n,k}^{S \circ}}
=  \sum_{r=k}^{n-1} (-1)^{r-k}\binom{r}{k} \cdot q^S_{n,r} \/.
\end{equation}
The Chern-Mather classes are given as follows. When $A=2k$ is even we have
\begin{equation}
c_M^{\tau^S_{A,B}}=\sum_{r=k}^{\lfloor \frac{B-1}{2} \rfloor} \binom{r}{k}  \cdot \left( \sum_{i=2r}^{B-1} (-1)^{i-2r}\binom{i}{2r} \cdot q^S_{B,i} \right)
+ \sum_{r=k}^{\lfloor \frac{B-2}{2} \rfloor} \binom{r}{k}  \cdot \left( \sum_{i=2r+1}^{B-1} (-1)^{i-2r-1}\binom{i}{2r+1} \cdot q^S_{B,i} \right) \/.
\end{equation}
When $A=2k+1$ is odd, we have 
\begin{equation}
c_M^{\tau^S_{A,B}}= 
\sum_{r=k}^{\lfloor \frac{B-2}{2} \rfloor} \binom{r}{k}  \cdot \left( \sum_{i=2r+1}^{B-1} (-1)^{i-2r-1}\binom{i}{2r+1} \cdot q^S_{B,i} \right)
\/.
\end{equation}
For skew-symmetric rank loci, we define $E_i$ to be the Euler numbers appearing as the coefficients of the Taylor expansion 
\[
\frac{1}{cosh(x)}=\sum_{n=0}^\infty \frac{E_n}{n!} x^n \/.
\]
For $k\geq 1$ we then have:
\begin{align}
c_{sm}^{\tau^{\wedge \circ}_{A,B}}
=& 
\begin{cases}
\sum_{r=k}^{n-1} \binom{2r}{2k}E_{2r-2k} \cdot q^\wedge_{2n,2r} & A=2n, B=2k \\
\sum_{r=k}^{n-1} \binom{2r+1}{2k+1}E_{2r-2k} \cdot q^\wedge_{2n+1,2r+1} & A=2n+1, B=2k+1  
\end{cases}
\\
c_{M}^{\tau^{\wedge}_{A,B}}
=& 
\begin{cases}
\sum_{r=k}^{n-1} \sum_{i=r}^{n-1} \binom{r}{k} \binom{2i}{2r}E_{2i-2r}\cdot  q^\wedge_{2n,2i} & A=2n, B=2k \\
\sum_{r=k}^{n-1} \sum_{i=r}^{n-1} \binom{r}{k} \binom{2i+1}{2r+1}E_{2i-2r}\cdot  q^\wedge_{2n+1,2i+1}   & A=2n+1, B=2k+1  
\end{cases}
\end{align}
\end{theo}
\begin{proof}[Proof of the Theorem]
Recall that 
for all three cases, set $*=\emptyset$, $*=\wedge$ and $*=S$, and set $\PP^N$ by $\PP(M_n)$, $\PP(M^\wedge_n)$ and $\PP(M^S_n)$
we have commutative diagrams of Tjurina transforms: 
\[
\begin{tikzcd}
 & \hat{\tau}^*_{n,k} \arrow{r}{} \arrow{d}{p} \arrow{dl}{q} & G(k,n) \times \PP^{N} \arrow{d} \\
G(k,n) & \tau^*_{n,k} \arrow{r}{} & \PP^{N}.
\end{tikzcd}
\]
The first projection $p$ is a resolution of singularity,  and is isomorphic over $\tau^{* \circ}_{n,k}$. The second projections $q$ identifies  
the Tjurina transforms with  projectivized bundles:
\[
\hat{\tau}_{n,k}\cong \PP(Q^{\vee n});\quad  \hat{\tau}^\wedge_{n,k}\cong \PP(\wedge^2 Q^\vee); \quad \hat{\tau}^S_{n,k}\cong \PP(Sym^2 Q^\vee) \/.
\]

First we show that $q^*_{n,k}$ polynomials are exactly the pushforward of the classes $p_*(c_{sm}^{\hat{\tau}^*_{n,k}})$ in the projective spaces $\PP(M^*_n)$. 
Write $p_*(c_{sm}^{\hat{\tau^*_{n,k}}})=\sum_{l} \gamma^*_l H^l \in A_*(\PP(M^*_n))$, and denote $N^*=\dim \PP(M^*_n)$.
The coefficients $\gamma^*_l$ thus can be computed as
$\gamma^*_l = \int_{\PP(M^*_n)} H^{N^*-l}\cap p_*(c_{sm}^{\hat{\tau^*_{n,k}}})$. Notice that the pull back of the hyperplane  bundle
$\cO_{\PP(M^*_n)}(1)$ on $\PP(M^*_n)$ to $\hat\tau^*_{n,k}$ agrees with the
tautological line bundle $\cO_{\hat{\tau}^*_{n,k}}(1)$, thus we denote  $\cO(1)$ for both of them.
Since $\int_X \alpha =\int_Y f_*\alpha$ for any class $\alpha$ and any proper morphism 
$f\colon X\to Y$, by the projection formula we have (omitting the obvious pullbacks):
\begin{align*}
\gamma^*_l
=& \int_{\PP(M^*_n)} H^{N^*-l} \cap p_*(c_{sm}^{\hat{\tau}^*_{n,k}}) 
= \int_{\hat{\tau}^*_{n,k}} c_1(\cO(1))^{N^*-l}\cap c_{sm}^{\hat{\tau}^*_{n,k}} \\
=& \int_{\hat{\tau}^*_{n,k}} c_1(\cO(1))^{N^*-l}c(\cT_{\hat{\tau}^*_{n,k}})\cap [\hat{\tau}^*_{n,k}]  \\
=& \int_{\hat{\tau}^*_{n,k}} c(S^\vee\otimes Q) c(E_*\otimes\cO(1)) c_1(\cO(1))^{N^*-l}\cap [\hat{\tau}^*_{n,k}] 
\end{align*}
Here $E_*$ denotes the vector bundles $Q^{\vee n}$, $Sym^2 Q^\vee$ and $\wedge^2 Q^\vee$ for three types of matrices respectively.
The last equation comes from the standard Euler sequence of projective bundle $\pi\colon \PP(E_*)\to X$:
\[
0\rightarrow \cO_{\PP(E_*)}(-1)\rightarrow \pi^*(E_*)\rightarrow T_{\PP(E_*)}\otimes\cO_{\PP(E_*)}(-1) \rightarrow 0 \/.
\]
Expand the tensor $c(E_*\otimes\cO(1))$ using \cite[Example 3.2.2]{INT}, and then combine the definition of Segre classes we have 
{\small
\begin{align*}
\gamma^*_l
=&  \int_{\hat{\tau}^*_{n,k}} \sum_{p=0}^{e^*}\sum_{i=0}^{p} 
\binom{e^*-i}{p-i}
c(S^\vee\otimes Q)c_i(E_*)c_1(\cO(1))^{N^*-l+p-i} \cap [\hat{\tau}^*_{n,k}] \\
=& 
\sum_{p=0}^{e^*}\sum_{i=0}^{p}   \binom{e^*-i}{p-i}
 \int_{G(k,n)} 
 c(S^\vee\otimes Q)c_i(E_*)s_{N^*-l+p-i+1-e^*}(E_*) \cap [G(k,n)]
\end{align*}
}
Here $e_*=\rk E_*$ are the ranks of the corresponding vector bundles. 

The rest computation of the Chern-Schwartz-MacPherson classes follows from \cite[Theorem 4.5 and 4.7]{PR19} and \cite{Xiping1}. The computation of the Chern-Mather classes follows from the knowledge of local Euler obstructions computed in \cite[Theorem 6.2, 6.4 and 6.6]{Xiping3}.
\end{proof}

Notice that to describe a polynomial function, instead of listing all the coefficients appeared, one can also list all the function values at integers. Thus here we give another description for the polynomials $q^*_{n,k}$ for $*$ being $\emptyset$, $\wedge$ and $S$.
\begin{theo}[\textbf{Equivalent formula II}]
\label{theo; formulaII}
Let $S$ and $Q$ be the universal sub and quotient bundles over the Grassmannian $G(r,n)$. 
We define $Q^\wedge(d)$ to be the following Chow (cohomology) classes $($we omit the obvious $\cap [G(r,n)]$ here$)$:
\begin{align*}
& Q_{n,r}(d)
:=  \left(\sum_{k=0}^{n(n-r)} (1+d)^{n(n-r)-k}c_k(Q^{\vee n})  \right) \left(\sum_{k=0}^{nr} d^{nr-k}c_k(S^{\vee n}) \right); \\
& Q^\wedge_{n,r}(d): =\\
& \left(\sum_{k=0}^{\binom{n-r}{2}} (1+d)^{\binom{n-r}{2}-k}c_k(\wedge^2 Q^\vee)  \right) 
\left(\sum_{k=0}^{\binom{r}{2}} d^{\binom{r}{2}-k} c_k(\wedge^2 S^\vee) \right)  
\left(\sum_{k=0}^{r(n-r)} d^{r(n-r)-k} c_k(S^\vee\otimes Q^\vee) \right) ; \\
& Q^S_{n,r}(d) :=\\
&\left( \sum_{k=0}^{\binom{n-r+1}{2}}   (1+d)^{\binom{n-r+1}{2}-k}c_k(Sym^2 Q^\vee)  \right)
 \left(\sum_{k=0}^{\binom{r+1}{2}} d^{\binom{r+1}{2}-k}  c_k(Sym^2 S^\vee) \right)  
  \left(\sum_{k=0}^{r(n-r)} d^{r(n-r)-k} c_k(S^\vee\otimes Q^\vee) \right) \/.
\end{align*} 
We have the following integration formulas : 
\begin{align*}
 q_{n,r}(d)
=& \int_{G(r,n)} c(S^\vee\otimes Q)\cdot Q_{n,r}(d)\cap [G(r,n)] - d^{n^2} \binom{n}{r};  \\
 q^\wedge_{n,r}(d)
=&  
\int_{G(r,n)} c(S^\vee\otimes Q)  \cdot Q^\wedge_{n,r}(d) \cap [G(r,n)] 
- d^{\binom{n}{2}} \binom{n}{r} \\
  q^S_{n,r}(d)
=& 
\int_{G(r,n)} c(S^\vee\otimes Q)  \cdot Q^S_{n,r}(d)  \cap [G(r,n)]
-d^{\binom{n+1}{2}} \binom{n}{r} \/. 
\end{align*}
\end{theo}

\begin{rema}
The polynomials $Q_{n,r}^*(d)$ can also be written in  virtual forms 
. Let $t$ be a `virtual variable' in the $K$ theory of $X$, i.e., a variable that can be substituted by any operation $t\colon K(X)\to K(X)$. 
For any vector bundle  $E$ of rank $e$ on $X$, we consider the `virtual tensor' $E\otimes t$, whose Chern class is expressed as 
\[
c(E\otimes t):=\prod_{k=0}^e (1+t)^{e-k}\cdot c_k(E) \/.
\]
The same notation is also used in \cite{ACT21}, in their recursive formulas of motivic Chern classes. 
This is equivalent to say that, the Chern roots of $E\otimes t$ are $$\{t+\alpha_1, t+\alpha_2, \cdots , t+\alpha_e\}\/,$$ providing that $\{\alpha_1,\alpha_2, \cdots , \alpha_e\}$ are the Chern roots of $E$. Then we can rewrite $Q^*_{n,r}(t)$ as
\begin{align*}
 Q_{n,r}(t)
&  :=   c(Q^{\vee n}\otimes t) c_{top}(S^{\vee n}) ; \\
Q^\wedge_{n,r}(t)&  : = 
  c (\wedge^2 Q^\vee\otimes t)  c_{top}(\wedge^2 S^\vee\otimes t)c_{top}(S^\vee\otimes Q^\vee\otimes t)  ; \\
 Q^S_{n,r}(t) &  :=
 c(Sym^2 Q^\vee\otimes t) c_{top}(Sym^2 S^\vee\otimes t)   c_{top}(S^\vee\otimes Q^\vee\otimes t)   \/.
\end{align*}
Here $c_{top}$ denotes the Chern classes of the top degrees. 
\end{rema}

\begin{proof}[Proof of Theorem~\ref{theo; formulaII}]
Recall that 
$q^*_{n,k}(H)=\sum_{l=0}^{N^*} \gamma^*_l H^l$ are defined as the pushforward $p_* c_{sm}^{\hat{\tau}^*_{n,k}}$. Here $N^*=\dim \PP(M^*_n)$ are the dimensions of the projective spaces. 
One then has $\gamma_l^*=\int_{\hat{\tau}^*_{n,k}}   c_{sm}^{\hat{\tau}^*_{n,k}}H^{N^*-l}$. This shows that 
\begin{align*}
q^*_{n,k}(\frac{1}{d})=& \sum_{l=0}^{N^*} \gamma^*_l d^{-l} = \sum_{l=0}^{N^*} \int_{\hat{\tau}^*_{n,k}}   c_{sm}^{\hat{\tau}^*_{n,k}}\cdot d^{-l}H^{N^*-l}
= \sum_{l=0}^{N^*} \int_{\hat{\tau}^*_{n,k}}   c_{sm}^{\hat{\tau}^*_{n,k}}\cdot d^{l-N^*}H^l\\
=& d^{-N^*}\cdot \sum_{l=0}^{N^*} \int_{\hat{\tau}^*_{n,k}}   c_{sm}^{\hat{\tau}^*_{n,k}}\cdot d^{l}H^l
= d^{-N^*}\cdot \int_{\hat{\tau}^*_{n,k}} \frac{c_{sm}^{\hat{\tau}^*_{n,k}}}{1-dH} \\
=& d^{-N^*}\cdot \int_{\PP(E^*)} \frac{c(S^\vee\otimes Q)c(E^*\otimes \cL)}{1-dH} \/.
\end{align*}
Here $E^*$ stands for $Q^{\vee n}$, $\wedge^2 Q^\vee$ and $Sym^2 Q^\vee$ when $*=\emptyset$, $*=\wedge$ and $*=S$ respectively. The vector bundles $S$ and $Q$ denote the universal sub and quotient bundles over the Grassmannian $G(k,n)$. To compute above integration we will need the following Lemma.
\begin{lemm}
\label{lemm; SegreTensor}
Let $E$ be a rank $e$ vector bundle over $X$, let $p\colon \PP(E)\to X$ be the projective bundle.
Let $\cL=\cO_{\PP(E)}(1)$ be the tautological bundle. We denote its Chern class $c_1(\cL)$ by $H$.  Then for any integer $d$ we have:
\[
d\cdot p_*\left(\frac{c(E\otimes \cL)}{1-d\cdot c_1(\cL)}  \right)
=
\left(\sum_{k=0}^e d^k(1+d)^{e-k}c_k(E)  \right) \left(\sum_{k=0}^\infty d^ks_k(E) \right) -1  \/.
\]
\end{lemm}
\begin{proof}
\begin{align*}
c(E\otimes \cL)=& 
\sum_{k=0}^e \left(\sum_{i=0}^k \binom{e-i}{k-i} c_i(E)\cdot c_1(\cL)^{k-i}\right)\\
=&
\sum_{k=0}^e \left(\sum_{j=k}^e \binom{e-j+k}{k} c_{j-k}(E) \right) H^k \\
=&
\sum_{k=0}^e \left(\sum_{j=0}^{e-k} \binom{e-j}{k} c_{j}(E) \right) H^k
\end{align*}
Thus for $ \frac{c(E\otimes \cL)}{1-d\cdot c_1(\cL)} $ we have
\begin{align*}
 \frac{c(E\otimes \cL)}{1-d\cdot c_1(\cL)} 
=& \sum_{l=0}^\infty  c(E\otimes \cL)\cdot d^lH^l= \sum_{l=0}^\infty \sum_{k=0}^e \left(\sum_{j=0}^{e-k} \binom{e-j}{k} c_{j}(E) \right) d^l H^{k+l} \/.
\end{align*}
Since we are pushing forward the Chern classes to the base $X$, by the definition of Segre class we only concern with $H^{\geq e-1}$ part. The coefficient for $H^{e-1}$ is 
\[
\sum_{k=0}^{e-1}  \left(\sum_{j=0}^{e-k}\binom{e-j}{k}c_j(E)\right) d^{e-1-k}=\frac{1}{d}\left( \sum_{k=0}^{e} d^{k}(1+d)^{e-k} c_k(E)-c_0(E)\right) ;
\]
and the coefficient for $H^{e+l}$, $l\geq 0$ is 
\[
\sum_{k=0}^e  \left(\sum_{j=0}^{e-k}\binom{e-j}{k}c_j(E)\right)  d^{e+l-k}=\sum_{k=0}^{e} d^{k+l}(1+d)^{e-k} c_k(E)
\]
Thus we have
\begin{align*}
d\cdot p_*\left(\frac{c(E\otimes \cL)}{1-d\cdot c_1(\cL)}  \right)
=&   \sum_{k=0}^{e} d^{k}(1+d)^{e-k} c_k(E)s_0(E)-c_0(E)s_0(E)  \\
+& \sum_{l\geq 0} \left(  \sum_{k=0}^{e} d^{k+l+1}(1+d)^{e-k} c_k(E)s_{l+1}(E) \right)  \\
=& \left(\sum_{k=0}^e d^k(1+d)^{e-k}c_k(E)  \right) \left(\sum_{k=0}^\infty d^ks_k(E) \right) -1 
\end{align*}
Notice that
although in the expression we have $\sum_{k=0}^\infty d^ks_k(E)$, this is actually a finite sum. When the degree of the Segre class exceeds the dimension of $X$, it then equals $0$. 
\end{proof}

Back to our case: the base space $X=G(r,n)$ is the Grassmannian.
For the ordinary rank loci $*=\emptyset$, the vector bundle $E_*=Q^{\vee n}$ has rank $n(n-r)$, and the ambient space $\PP(M_n)$ has dimension $N=n^2-1$. Thus we have
\begin{align*}
d^{n^2}\cdot q_{n,r}(\frac{1}{d})
=& d\cdot d^{n^2-1}\cdot q_{n,r}(\frac{1}{d})
=d\cdot \int_{\PP(Q^{\vee n})} \frac{c(S^\vee\otimes Q)c(Q^{\vee n}\otimes \cL)}{1-d\cdot c_1(\cL)}  \\
=& \int_{G(r,n)} c(S^\vee\otimes Q)  \left(\sum_{k=0}^{n(n-r)} d^k(1+d)^{n(n-r)-k}c_k(Q^{\vee n})  \right) \left(\sum_{k=0}^\infty d^ks_k(Q^{\vee n}) \right)  - \binom{n}{r} \\
=& \int_{G(r,n)} c(S^\vee\otimes Q)  \left(\sum_{k=0}^{n(n-r)} d^k(1+d)^{n(n-r)-k}c_k(Q^{\vee n})  \right) \left(\sum_{k=0}^{nr} d^kc_k(S^{\vee n}) \right)  -  \binom{n}{r} 
\end{align*}
Substitute $d$ by $d^{-1}$ we have
\[
q_{n,r}(d)= \int_{G(r,n)} c(S^\vee\otimes Q)  \left(\sum_{k=0}^{n(n-r)} (1+d)^{n(n-r)-k}c_k(Q^{\vee n})  \right) \left(\sum_{k=0}^{nr} d^{nr-k}c_k(S^{\vee n}) \right)  - d^{n^2} \binom{n}{r} \/.
\]

For the skew-symmetric rank loci $*=\wedge$, the bundle $E^*=\wedge^2 Q^\vee$ is of rank $\binom{n-r}{2}$ and we have $N^\wedge=\binom{n}{2}-1$. Thus one obtains
\begin{align*}
d^{\binom{n}{2}}\cdot q^\wedge_{n,r}(\frac{1}{d})
=& d\cdot d^{\binom{n}{2}-1} q^\wedge_{n,r}(d) 
= d\cdot \int_{\PP(\wedge^2 Q^\vee)} \frac{c(S^\vee\otimes Q)c(\wedge^2 Q^\vee\otimes \cL)}{1-d\cdot c_1(\cL)}  \\
=& \int_{G(r,n)} c(S^\vee\otimes Q)  \left(\sum_{k=0}^{\binom{n-r}{2}} d^k(1+d)^{\binom{n-r}{2}-k}c_k(\wedge^2 Q^\vee)  \right) \left(\sum_{k=0}^\infty d^ks_k(\wedge^2 Q^\vee) \right) -\binom{n}{r} 
\end{align*}
Substitute $d$ by $d^{-1}$ we then have
\[
q^\wedge_{n,r}(d)= \int_{G(r,n)} c(S^\vee\otimes Q)  
\left(\sum_{k=0}^{\binom{n-r}{2}} (1+d)^{\binom{n-r}{2}-k}c_k(\wedge^2 Q^\vee)  \right) \left(\sum_{k=0}^{\infty} d^{A_r-k} s_k(\wedge^2 Q^\vee) \right)  
- d^{\binom{n}{2}} \binom{n}{r} \/.
\]
Here we take $A_r=\binom{n}{2}-\binom{n-r}{2}=\binom{r}{2}+r(n-r)$.  Notice that we have 
\[
c(\wedge^2 Q^\vee)c(S^\vee\otimes Q^\vee)c(\wedge^2 S^\vee))=1;\quad A+r=\binom{r}{2}+r(n-r) \/.
\]

Define $Q^\wedge_{n,r}(d)$ to be the following Chow (cohomology) class
\[
\left(\sum_{k=0}^{\binom{n-r}{2}} (1+d)^{\binom{n-r}{2}-k}c_k(\wedge^2 Q^\vee)  \right) 
\left(\sum_{k=0}^{\binom{r}{2}} d^{\binom{r}{2}-k} c_k(\wedge^2 S^\vee) \right)  
\left(\sum_{k=0}^{r(n-r)} d^{r(n-r)-k} c_k(S^\vee\otimes Q^\vee) \right)   \/,
\]
then the formula can be written as
\[
 q^\wedge_{n,r}(d) = 
  \int_{G(r,n)} c(S^\vee\otimes Q)\cdot Q^\wedge_{n,r}(d)-  d^{\binom{n }{2}} \binom{n}{r}
\/.
\]

For the symmetric rank loci $*=S$, $E^*=Sym^2 Q^\vee$ is of rank $\binom{n-r+1}{2}$ and $N^S=\binom{n+1}{2}-1$. Thus we have
\begin{align*}
&d^{\binom{n+1}{2}}\cdot q^S_{n,r}(\frac{1}{d})
=  d\cdot d^{\binom{n+1}{2}-1} q^S_{n,r}(d) 
= d\cdot \int_{\PP(Sym^2 Q^\vee)} \frac{c(S^\vee\otimes Q)c(Sym^2 Q^\vee\otimes \cL)}{1-d\cdot c_1(\cL)}  \\
=& \int_{G(r,n)}  c(S^\vee\otimes Q)  \left(\sum_{k=0}^{\binom{n-r+1}{2}} d^k(1+d)^{\binom{n-r+1}{2}-k}c_k(Sym^2 Q^\vee)  \right) \left(\sum_{k=0}^\infty d^ks_k(Sym^2 Q^\vee) \right) -\binom{n}{r} 
\end{align*}
Substitute $d$ by $d^{-1}$ we then have
\begin{align*}
&q^S_{n,r}(d)+  d^{\binom{n+1}{2}} \binom{n}{r} \\
=& \int_{G(r,n)} c(S^\vee\otimes Q)  
\left(\sum_{k=0}^{\binom{n-r+1}{2}} (1+d)^{\binom{n-r+1}{2}-k}c_k(Sym^2 Q^\vee)  \right) \left(\sum_{k=0}^{\infty} d^{B_r-k} s_k(Sym^2 Q^\vee) \right)  
\/.
\end{align*}
Here we take $B_r=\binom{n+1}{2}-\binom{n-r+1}{2}$.
Notice that we have 
\[
c(Sym^2 Q^\vee)c(S^\vee\otimes Q^\vee)c(Sym^2 S^\vee))=1; \quad B_r=\binom{r+1}{2}+r(n-r)\/.
\]
Define $Q^S_{n,r}(d)$ to be the following Chow (cohomology) class
\[
\left(\sum_{k=0}^{\binom{n-r+1}{2}} (1+d)^{\binom{n-r+1}{2}-k}c_k(Sym^2 Q^\vee)  \right)
 \left(\sum_{k=0}^{\binom{r+1}{2}} d^{\binom{r+1}{2}-k}  c_k(Sym^2 S^\vee) \right)  
  \left(\sum_{k=0}^{r(n-r)} d^{r(n-r)-k} c_k(S^\vee\otimes Q^\vee) \right) \/,
\]
then the formula can be written as
\[
 q^S_{n,r}(d) = 
  \int_{G(r,n)} c(S^\vee\otimes Q)\cdot Q^S_{n,r}(d)-  d^{\binom{n+1}{2}} \binom{n}{r}
\/.
\]
This complete the proof of the Theorem.
\end{proof}

Recall that for a projective variety, Aluffi's $\cJ$ involution interchanges the Chern-Schwartz-MacPherson $\gamma$ polynomial and the sectional Euler characteristics polynomial. Here the sectional Euler characteristic polynomial $\chi_X(t)$ is defined as follows: $\chi_X(t):=\sum_{k\geq 0} \chi(X\cap L^k)\cdot (-t)^k$ for $L^k$ being a generic codimension $k$ linear subspace. 
For generic determinantal varieties
We define the  $\Gamma$ polynomials as follows. 
\begin{align*}
d\cdot \Gamma_{n,r}(d)=& 
\int_{G(r,n)} c(S^\vee\otimes Q) \left(\sum_{k=0}^{n(n-r)} d^k(1+d)^{n(n-r)-k}c_k(Q^{\vee n})  \right)\left(\sum_{k=0}^\infty d^k s_k(Q^{\vee n}) \right)-\binom{n}{r} \\
d\cdot \Gamma^\wedge_{n,r}(d)=& 
\int_{G(r,n)}c(S^\vee\otimes Q) \left(\sum_{k=0}^{\binom{n-r}{2}} d^k(1+d)^{\binom{n-r}{2}-k}c_k(\wedge^2 Q^\vee)  \right)\left(\sum_{k=0}^\infty d^k s_k(\wedge^2 Q^\vee) \right) -\binom{n}{r}\\
d\cdot \Gamma^S_{n,r}(d)=&
\int_{G(r,n)} c(S^\vee\otimes Q) \left(\sum_{k=0}^{\binom{n-r+1}{2}} d^k(1+d)^{\binom{n-r+1}{2}-k}c_k(Sym^2 Q^\vee)  \right)\left(\sum_{k=0}^\infty d^k s_k(Sym^2 Q^\vee) \right)  -\binom{n}{r} \/.
\end{align*}
The $\Gamma$ polynomials are related to $q$ polynomials by $d\mapsto d^{-1}$, since the Chern-Schwartz-MacPherson $\gamma$ polynomials are related to Chern-Schwartz-MacPherson classes by $H\to H^{-1}$. We have the following result.
\begin{coro}
For any integer $d$, following the proof in Formula II we have:
\begin{align*}
 \chi_{\tau_{n,k}^{\circ}}(d)
=& \sum_{r=k}^{n-1} (-1)^{r-k}\binom{r}{k} \cdot \frac{d\cdot \Gamma_{n,r}(-1-d)+\Gamma_{n,r}(0)}{1+d} \\
 \chi_{\tau_{n,k}^{S \circ}}(d)
=&  \sum_{r=k}^{n-1} (-1)^{r-k}\binom{r}{k} \cdot  \frac{d\cdot \Gamma^S_{n,r}(-1-d)+\Gamma^S_{n,r}(0)}{1+d}  \\
\chi_{\tau^{\wedge \circ}_{A,B}}(d)
=& 
\begin{cases}
\sum_{r=k}^{n-1} \binom{2r}{2k}E_{2r-2k} \cdot   \frac{d\cdot \Gamma_{2n,2r}(-1-d)+\Gamma^\wedge_{2n,2r}(0)}{1+d} & A=2n, B=2k \\
\sum_{r=k}^{n-1} \binom{2r+1}{2k+1}E_{2r-2k} \cdot \frac{d\cdot \Gamma_{2n+1,2r+1}(-1-d)+\Gamma^\wedge_{2n+1,2r+1}(0)}{1+d} & A=2n+1, B=2k+1  
\end{cases}
\end{align*}
\end{coro}
\begin{proof}
The proof is a direct application of Aluffi's involution formula. The evaluations here is valid due to the fact that $\frac{t\cdot f(-1-t)+f(0)}{1+t}$ is actually a polynomial for any $f(t)$, instead of the truncation of the first $N$ terms from an infinite power series. 
\end{proof}

\section{Characteristic Cycles and Polar Degrees}
\label{S; charcycle}
In this section we take $K=\CC$. In complex category, the
theory of Chern classes can be thought of the pushdown of the theory of characteristic cycles of constructible sheaves.  
Consider the embedding $i\colon X\subset M$ of  a $d$-dimensional  variety into a $m$-dimensional complex manifold. The \textit{conormal space} of $X$ is defined as the dimension $m$ subvariety of $T^*M$:
\[
T^*_X M:=\overline{\{(x,\lambda)|x\in X_{sm};\lambda(T_xX)=0\}}\subset T^*M
\]
This is a conical Lagrangian subvariety of $T^*M$. In fact, the conical Lagrangian subvarieties of $T^*M$  supporting inside  $X$ are exactly the conormal spaces of closed subvarieties $V\subset X$. For a  proof we refer to \cite[Lemma 3]{Kennedy}.
Let $L(X)$ be the free abelian group generated by the conormal spaces $T_V^*M$ for subvarieties $V\subset X$, and we call an element of $L(M)$ a \textit{(conical) Lagrangian cycle} of $X$. We say a Lagrangian cycle is \textit{irreducible} if it equals the conormal space of some subvariety $V$. 

The group $L(X)$ is independent of the embedding: the group $L(X)$ is isomorphic to the group of constructible functions $F(X)$ by the group morphism $Eu$ that sends $(-1)^{\dim V} T^*_V M$ to $Eu_V$. 
However,  the fundamental classes $[T^*_X M]$ depend  on the Chow ring of the ambient space. When the embedding $M$ is specified, we call
$[T^*_X M]\in A_*(T^*M)$ the \textit{Conormal cycle class of $X$} in $M$.
We define the \textit{projectivized conormal cycle class} of $X$ to be $Con(X):=[\PP(T^*_X M)]$, which is a $m-1$-dimensional cycle in the total space $\PP(T^*M)$.

Composing the two operations we obtain  a group homomorphism 
$$
Ch\colon F(X)\to  A_{m-1}(\PP(T^*M))
$$
sending $Eu_V$ to $(-1)^{\dim V} Con(V)$. The cycle class $Ch(\id_X)$ is called the \textit{Characteristic Cycle class of $X$}, and denoted by $Ch(X)$. The `casting the shadow ' process discussed in \cite{Aluffi04} relates the $Ch(\id_X)$ with $c_{sm}^X$, and $Ch(Eu_X)$ with $c_{M}^X$. 

\begin{prop}
\label{prop; charcycle}
Let $X^*_{n,k}\subset \PP^N$ be an EIDV of type $*$, for $*$ being $\emptyset$, $S$ or $\wedge$. Let \[
c_M^{X^*_{n,k}}=\sum_{l=0}^{N} \beta_l H^{N-l}; \quad
c_{sm}^{X^*_{n,k}}=\sum_{l=0}^{N} \gamma_l H^{N-l}
\]
be the Chern-Mather class and Chern-MacPherson-Schwartz class in $A_*(\PP^{N})$ respectively, as computed in \S\ref{S; charclass}. 
Let $d^*_{n,k}$ be the dimension of $X^*_{n,k}$,  
then the projectivized conormal cycle $Con(X^*_{n,k})$ equals:
\[
Con(X^*_{n,k})=  (-1)^{d^*_{n,k}}\sum_{j=1}^{N-1} \sum_{l=j-1}^{N-1} (-1)^l\beta_l\binom{l+1}{j} h_1^{N+1-j}h_2^{j}\cap [\PP^{N} \times \PP^{N}] \/.
\]
The characteristic cycle of $X^*_{n,k}$  are given by
\[
Ch(X^*_{n,k})=(-1)^{d^*_{n,k}}\sum_{j=1}^{N-1} \sum_{l=j-1}^{N-1} (-1)^l\gamma_l\binom{l+1}{j} h_1^{N+1-j}h_2^{j}\cap [\PP^{N} \times \PP^{N}] ;
\]
\end{prop}
\begin{proof}
Firstly, note that when $M=\PP^N$ we have the following diagram
\[
\begin{tikzcd}
P=\PP(T^* M) \arrow{r}{j} \arrow{d}{\pi} & \PP^{N}\times\PP^{N} \arrow{dl}{pr_1}\arrow{d}{pr_2} \\
M=\PP^{N} & (\PP^{N})^\vee=M^*
\end{tikzcd} .
\]
Here $P$ is embedded as the incidence variety.
Let $L_1,L_2$ are the pull backs of the line bundle $\cO_{\PP^{N}}(1)$ of $\PP^N$ from projections $pr_1$ and $pr_2$. Then we have 
$\cO_P(1)=j^*(L_1\otimes L_2)$, and $j_*[\PP(T^* M)]=c_1(L_1\otimes L_2) \cap [\PP^{N}\times\PP^{N}]$ is a divisor in $\PP^{N}\times\PP^{N}$. Thus both the characteristic cycle and the conormal cycles can be realized as polynomials in $h_1=c_1(L_1)$ and $h_2=c_1(L_2)$, as
classes in $A_*(\PP^{N}\times\PP^{N})$.

For any constructible function $\varphi\in F(X)$,  we define the signed class $\breve{c}_*(\varphi)\in A_*(\PP^{N})$  as
$\{\breve{c}_*(\varphi)\}_r = (-1)^r\{c_*(\varphi)\}_r $. Here for any class $C\in A_*(M)$, $C_r$ denotes the $r$-dimensional piece of $C$. 
As proved in \cite[Lemma 4.3]{Aluffi04}, this class is exactly the shadow of the characteristic cycle $Ch(\varphi)$. 
For $i=1,2$, let $h_i=c_1(L_i)\cap [\PP^{N}\times\PP^{N}] $ be the pull backs of hyperplane classes.
Write $c_*(\varphi)=\sum_{l=0}^{N} \gamma_lH^{N-l}$ as a polynomial of $H$, then by the structure theorem for projective bundles we have inversely:
\[
Ch(\varphi)=\sum_{j=1}^{N}\sum_{k=j-1}^{N-1} (-1)^k\gamma_k\binom{k+1}{j} h_1^{N+1-j}h_2^j 
\]
as a class in $\PP^{N}\times\PP^{N}$. Set $\varphi$ to be $\id_X$ and $Eu_X$ one obtains the proposition.
\end{proof}

Proved in \cite{Piene15}\cite[Remark 2.7]{Aluffi04}, the multiplicities appeared in  the expression of the projectivized conormal cycle $Con(X)$ are exactly the polar degrees of $X$. Write $c_M^{X^*_{n,k}}=\sum_{l=0}^{N} \beta_l H^{N-l}$, then we obtain a formula for the polar degrees of $X^*_{n,k}$:
\begin{equation}
\label{eq; polar}
P_j=(-1)^{d^*_{n,k}} \sum_{l=j-1}^{N-1} (-1)^l\beta_l\binom{l+1}{j} \/.
\end{equation}
The sum of the polar degrees is also a very interesting invariant. It is called the generic Euclidean distance degree of $X$, and denoted by $gED(X)$. We refer to \cite{AC18} for more details. The generic Euclidean distance degree of $X^*_{n,k}$ is given by
\[
gED(X^*_{n,k})=\sum_{l=0}^{d^*_{n,k}}\sum_{i=0}^{l} (-1)^{i}\binom{d^*+1-i}{d^*+1-l}\beta_{d^*-i}  \/.
\]

We define the following `flip' operation in $A_{n-1}(\PP^N\times \PP^N)$. For any class
 $\alpha=\sum_{i=0}^n \delta_i h_1^ih_2^{n-i}$, its flip
 $\alpha^\dagger$ is defined as $\alpha^\dagger:=\sum_{i=0}^n \delta_i h_1^{n-i}h_2^i$. In other word, we just switch the powers of $h_1$ to $h_2$. 
This `flip' process is compatible with addition: $(\alpha+\beta)^\dagger=\alpha^\dagger + \beta^\dagger$. Aluffi's projective duality involution shows that
\begin{prop}
For any projective subvariety $X\subset \PP^N$ with dual variety $X^\vee$ 
we have $Con(X^\vee)=Con(X)^\dagger$. 
Moreover, one can see that the $l$-th polar degree of $X$ equals the $(\dim X-l)$-th polar degree of $X^\vee$, and hence $gED(X)=gED(X^\vee)$.
\end{prop}
In particular, for generic determinantal varieties we have the following symmetry proposition.
\begin{prop}
\label{prop; charcyclesymmetry}
The characteristic cycles of $\tau^S_{n,1}$, $\tau^\wedge_{2n,2}$ and $\tau^\wedge_{2n+1,3}$ are symmetric:
\[
 Ch(\tau^S_{n,1})=Ch(\tau^S_{n,1})^\dagger; \quad Ch(\tau^\wedge_{2n,2})=Ch(\tau^\wedge_{2n,2})^\dagger; \quad Ch(\tau^\wedge_{2n+1,3})=Ch(\tau^\wedge_{2n+1,3})^\dagger \/.
\] 
\end{prop}
\begin{proof}
First we prove for skew-symmetric case. 
Recall that $
Ch\colon F(X)\to  A_{m-1}(\PP(T^*M))
$
sends $Eu_V$ to $(-1)^{\dim V} Con(V)$.
Thus from Theorem~\ref{theo; formulaI} have
\[
Ch(\tau^\wedge_{2n,2})=\sum_{i=1}^{n-1} (-1)^{i+1}\cdot (-1)^{2i(2n-2i)+\binom{2n-2i}{2}-1} Con(\tau^\wedge_{2n,2i})
=\sum_{i=1}^{n-1} (-1)^{n-1} Con(\tau^\wedge_{2n,2i}).
\]
We have shown that $Con(\tau_{m,n,i})=Con(\tau_{m,n,n-i})^\dagger$, thus 
\begin{align*}
\left(Con(\tau^\wedge_{2n,2i})+Con(\tau^\wedge_{2n,2n-2i})\right)^{\dagger}=Con(\tau^\wedge_{2n,2i})^\dagger +Con(\tau^\wedge_{2n,2n-2i})^\dagger =Con(\tau_{m,n,n-i})+Con(\tau_{m,n,i}).  
\end{align*}
is symmetric. Thus we have
\begin{align*}
&Ch(\tau^\wedge_{2n,2})
=\sum_{i=1}^{n-1} (-1)^{n-1} Con(\tau_{m,n,i}) \\
=&(-1)^{n-1} \left(Con(\tau_{m,n,1})+Con(\tau_{m,n,n-1}) + Con(\tau_{m,n,2})+Con(\tau_{m,n,n-2})+\cdots \right). 
\end{align*}
is a sum of symmetric terms, and hence is symmetric.
The proof for $Ch(\tau^\wedge_{2n+1,3})$ and $Ch(\tau^S_{n,1})$ follows from the same argument, by computing the base change between indicator functions and Euler obstruction functions using
Equation $(1)(2)$ and $(5)(6)$ in
Theorem~\ref{theo; formulaI}.
\end{proof}

\section{Conjecture}
\label{conj}
We close this paper with the following  conjectures:
\begin{conj}[Positivity]
All the coefficients appeared in $c_{sm}^{\tau_{n,k}^{\circ}}$, $c_{sm}^{\tau_{n,k}^{\wedge \circ}}$ and $c_{sm}^{\tau_{n,k}^{S \circ}}$ are non-negative. 
\end{conj}
This was proved for Schubert cells in flag manifold in \cite{AMSS17}. We don't know a proof for the determinantal varieties.
\begin{conj}[Log Concave]
For $*$ being $\emptyset$, $\wedge$ and $S$,  the coefficients appeared in $c_{sm}^{\tau_{n,k}^{* \circ}}$, $Con(\tau^*_{n,k})$ and $Ch(\tau^*_{n,k})$ are log concave. 
\end{conj}

\section{Appendix: Examples of Chern Classes}
\label{S; Appendix}
\subsection{Skew-Symmetric Matrix}
\subsubsection{$n=6$}
The total space is $\PP(M_6^\wedge)=\PP^{14}$.
\begin{align*}
q^\wedge_{6,2}
=& 90H^{14}   + 405H^{13}   + 1290H^{12}   + 2925H^{11}   + 4878H^{10}   + 6225H^{9}  + 6318H^{8}  + 5217H^7 \\
& + 3504H^6  + 1863H^5  + 744H^4  + 207H^3  + 36H^2+3H \\
q^\wedge_{6,4}
=& 15H^{14}   + 60H^{13}   + 170H^{12}   + 330H^{11}   + 438H^{10}   + 394H^9  + 234H^8 + 84H^7  + 14H^6
\end{align*}
Thus 
\begin{align*}
c_{sm}^{\tau^{\wedge \circ}_{6,0}}
=& q^\wedge_{6,2}-6q^\wedge_{6,4}\\
=& 15H^{12}+90H^{11}+315H^{10}+750H^9+1287H^8+1638H^7+1571H^6+1140H^5+621H^4\\
&+248H^3+69H^2+12H+1\\
c_{sm}^{\tau^{\wedge \circ}_{6,2}}
=& 45H^{13} + 270H^{12}   + 945H^{11}   + 2250H^{10}   + 3861 H^{9}  + 4914H^{8}  + 4713H^7 \\
& + 3420H^6  + 1863H^5  + 744H^4  + 207H^3  + 36H^2+3H \\
c_{sm}^{\tau^{\wedge }_{6,4}}
=& 15H^{14}   + 60H^{13}   + 170H^{12}   + 330H^{11}   + 438H^{10}   + 394H^9  + 234H^8 + 84H^7  + 14H^6
\end{align*}
One can observe that 
\[
3H c_{sm}^{\tau^{\wedge \circ}_{6,0}}=c_{sm}^{\tau^{\wedge \circ}_{6,2}} \/.
\]
The characteristic cycles and conormal cycles are computed as:
\vskip .05in
\begin{tabular}{c | *{8}{c}}
 Table & $h_1^{14}h_2$ & $h_1^{13}h_2^2$ & $h_1^{12}h_2^3$ & $h_1^{11}h_2^4$ & $h_1^{10}h_2^5$ & $ h_1^{9}h_2^6 $  & $h_1^8h_2^7$&$ h_1^7h_2^8$ \\
\hline
$Ch(\tau^\wedge_{6,2})$ & -3  & -6 & -12 & -24 & -48 & -82 & -108 & -108 \\
$Ch(\tau^\wedge_{6,4})=Con(\tau^\wedge_{6,4})$ & 3  & 6 & 12 & 24 & 48 & 68 & 66 & 42 \\
$Con(\tau^\wedge_{6,2})$ & 0  & 0 & 0 & 0 & 0 & -14 & -42 & -66 \\
\end{tabular}
\vskip .05in
\begin{tabular}{c | *{6}{c}}
 Table & $h_1^6h_2^9$ & $h_1^5h_2^{10}$ & $h_1^4h_2^{11}$ & $h_1^3h_2^{12}$ & $h_1^2h_2^{13} $ & $h_1h_2^{14}$   \\
\hline
$Ch(\tau^\wedge_{6,2})$  & -82 & -48 & -24  & -12 & -6  & -3   \\
$Ch(\tau^\wedge_{6,4})=Con(\tau^\wedge_{6,4})$  &   14 &  0  &  0  &  0 &  0 &  0  \\
$Con(\tau^\wedge_{6,2})$ &   -68 &  -48  & -24  & -12 & -6  & -3    \\
\end{tabular}\\
One can observe the duality in $Con(\tau^\wedge_{6,2})$ and $Con(\tau^\wedge_{6,4})$, since they are projective dual to each other. One can also observe the symmetry of $Ch(\tau^\wedge_{6,2})$, as proved in Proposition~\ref{prop; charcyclesymmetry}.

\subsubsection{n=7}
The total space is $\PP(M_7^\wedge)=\PP^{20}$.
\begin{align*}
q^\wedge_{7,3}
=& 210H^{20}   + 1155H^{19}   + 4690H^{18}   + 14175H^{17}   + 32970H^{16}   + 61299H^{15}   + 94698H^{14}  \\
+& 125139H^{13} +142898H^{12} + 139839H^{11}+ 115038H^{10}   + 77777H^9  + 42238H^8  \\
+& 17965H^7  + 5782H^6  + 1330H^5 +196H^4  + 14H^3\\
q^\wedge_{7,5}
=&  21H^{20}   + 105H^{19}   + 385H^{18}   + 1015H^{17}  + 1939H^{16}   + 2695H^{15}   + 2719H^{14}   + 1960H^{13}   \\
+& 966H^{12}   + 294H^{11}   + 42H^{10}
\end{align*}
Thus we have
\begin{align*}
c_{sm}^{\tau^{\wedge \circ}_{7,1}}
=& 105H^{18}+ 945H^{17}+  4830H^{16}+  17220H^{15}+  46053H^{14}+  95991H^{13}+  159726H^{12}\\
+&  215523H^{11}+  238056H^{10}+  216153H^9+  161252H^8+ 98315H^7+ 48482H^6+ 19019H^5\\
+& 5789H^4+ 1327H^3+ 210H^2+21H+1
\\
c_{sm}^{\tau^{\wedge \circ}_{7,3}}
=& q^\wedge_{7,3}-10q^\wedge_{7,5} \\
=& 105H^{19}+ 840H^{18}+ 4025H^{17}+ 13580H^{16}+ 34349H^{15}+ 67508H^{14}+ 105539H^{13}\\
+& 133238H^{12}+ 136899H^{11}+ 114618H^{10}+77777H^9 + 42238H^8  + 17965H^7+ 5782H^6  \\
+& 1330H^5  + 196H^4  + 14H^3 \\
c_{sm}^{\tau^{\wedge }_{7,5}}
=& 21H^{20}   + 105H^{19}   + 385H^{18}   + 1015H^{17}  + 1939H^{16}   + 2695H^{15}   + 2719H^{14}   + 1960H^{13}   \\
+& 966H^{12}   + 294H^{11}   + 42H^{10}
\end{align*}
The characteristic cycles and conormal cycles are computed as:
\vskip .05in
\begin{tabular}{c | *{10}{c}}
 Table & $h_1^{20}h_2$ & $h_1^{19}h_2^2$ & $h_1^{18}h_2^3$ & $h_1^{17}h_2^4$ & $h_1^{16}h_2^5$ & $ h_1^{15}h_2^6 $  & $h_1^8{14}h_2^7$&$ h_1^{13}h_2^8$ & $ h_1^{12}h_2^9$ & $ h_1^{11}h_2^{10}$  \\
\hline
$Ch(\tau^\wedge_{7,3})$ & 0 & 0 & -14 & -56 & -140 & -266 & -395 & -434 & -336 & -210 \\
$Ch(\tau^\wedge_{7,5})=Con(\tau^\wedge_{7,5})$ & 0  & 0 & 14 & 56 & 140 & 266 & 395 & 434 & 336 & 168\\
$Con(\tau^\wedge_{7,3})$ & 0  & 0 & 0 & 0 & 0 & 0 & 0 & 0 & 0   & -42 \\
\end{tabular}
\vskip .05in
\begin{tabular}{c | *{10}{c}}
 Table & $h_1^8h_2^{13}$ & $h_1^7h_2^{14}$ & $h_1^6h_2^{15}$ & $h_1^5h_2^{16}$ & $h_1^4h_2^{17} $ & $h_1^3h_2^{18}$  &  $h_1^2h_2^{19}$ & $h_1h_2^{20}$ & $ h_1^{10}h_2^{11}$ &
 $ h_1^{9}h_2^{12}$ \\
\hline
$Ch(\tau^\wedge_{6,2})$ & -210 & -336  & -434 & -395  & -266 & -140  & -56 & -14  & 0  & 0 \\
$Ch(\tau^\wedge_{6,4})=Con(\tau^\wedge_{6,4})$    & 42   &  0 &  0  &  0  &  0 &  0 &  0 &  0&  0  &  0\\
$Con(\tau^\wedge_{6,2})$ &   -168 & -336  & -434 & -395  & -266 & -140  & -56 & -14  & 0  & 0   \\
\end{tabular}\\
One can observe the duality in $Con(\tau^\wedge_{7,3})$ and $Con(\tau^\wedge_{7,5})$, since they are projective dual to each other.  One can also observe the symmetry of $Ch(\tau^\wedge_{7,3})$  proved in Proposition~\ref{prop; charcyclesymmetry}.

\subsection{Symmetric Matrices}
\subsubsection{$n=3$}
The total space is $\PP^{5}$.
\begin{align*}
q^S_{3,1}
=& 9H^5  + 18H^4  + 18H^3  + 9H^2  + 3H \\
q^S_{3,2}
=&  3H^5  + 6H^4  + 4H^3 
\end{align*}
Thus we have
\begin{align*}
c_{sm}^{\tau^{S \circ}_{3,0}}
=& 3H^4+6H^3+6H^2+3H+1 \\
c_{sm}^{\tau^{S \circ}_{3,1}}
=& q^S_{3,1}-2q^S_{3,2}\\
=& 3H^5+6H^4+10H^3+9H^2  + 3H\\
c_{sm}^{\tau^{S}_{3,2}}
=&  3H^5  + 6H^4  + 4H^3 
\end{align*}
One can observe that
\[
3H\cdot c_{sm}^{\tau^{S \circ}_{3,0}}=c_{sm}^{\tau^{S \circ}_{3,1}}+2\cdot c_{sm}^{\tau^{S \circ}_{3,2}} \/.
\]
The characteristic cycles and conormal cycles are computed as:
\vskip .05in
\begin{tabular}{c | *{5}{c}}
 Table & $h_1^{5}h_2$ & $h_1^{4}h_2^2$ & $h_1^{3}h_2^3$ & $h_1^{2}h_2^4$ & $h_1h_2^5$   \\
\hline
$Ch(\tau^S_{3,1})$ & 3  & 6 & 8 & 6 & 3   \\
$Ch(\tau^S_{3,2})=Con(\tau^S_{3,2})$ & 3  & 6 & 4 & 0 & 0   \\
$Con(\tau^S_{3,1})$ & 0  & 0 & 4 & 6 & 3   \\
\end{tabular} \\
The symmetry of $Ch(\tau^S_{3,1})$ is proved in Proposition~\ref{prop; charcyclesymmetry}, and the duality of $Con(\tau^S_{3,1})$ and $Con(\tau^S_{3,2})$ come from projective duality.

\subsubsection{$n=4$}
The total space is $\PP(M_4^S)=\PP^9$.
\begin{align*}
q^S_{4,1}
=& 24H^9  + 84H^8  + 184H^7  + 264H^6  + 264H^5 + 184H^4  + 84H^3  + 24H^2  + 4H
\\
q^S_{4,2}
=&18H^9  + 54H^8  + 92H^7  + 96H^6  + 72H^5  + 40H^4  + 10H^3
\\
q^S_{4,3}
=& 4H^9  + 12H^8  + 16H^7  + 8H^6
\end{align*}
Thus we have
\begin{align*}
c_{sm}^{\tau^{S \circ}_{4,0}}
=&
3H^8+12H^7+34H^6+60H^5 +66H^4+46H^3+21H^2+6H+1 \\
c_{sm}^{\tau^{S \circ}_{4,1}}
=& q^S_{4,1}-2q^S_{4,2}+3q^S_{4,3} \\
=& 
12H^8+48H^7+96H^6+120H^5+104H^4+64H^3+24H^2  + 4H \\
c_{sm}^{\tau^{S \circ}_{4,2}}
=& q^S_{4,2}-3q^S_{4,3} \\
=& 6H^9+18H^8+44H^7+72H^6+72H^5  + 40H^4  + 10H^3 \\
c_{sm}^{\tau^{S }_{4,3}}
=& 4H^9  + 12H^8  + 16H^7  + 8H^6
\end{align*}
One can observe that
\[
4H\cdot c_{sm}^{\tau^{S \circ}_{4,0}}=c_{sm}^{\tau^{S \circ}_{4,1}}+2\cdot c_{sm}^{\tau^{S \circ}_{4,2}}\/.
\]
The characteristic cycles and conormal cycles are computed as:
\vskip .05in
\begin{tabular}{c | *{9}{c}}
 Table & $h_1^{9}h_2$ & $h_1^{8}h_2^2$ & $h_1^{7}h_2^3$ & $h_1^{6}h_2^4$ & $h_1^{5}h_2^5$ & $ h_1^{4}h_2^6 $  & $h_1^3h_2^7$&$ h_1^2h_2^8$ & $h_1h_2^9$ \\
\hline
$Ch(\tau^S_{4,1})$ & 4  & 12 & 26 & 38 & 42 & 38 & 26 & 12 & 4 \\
$Ch(\tau^S_{4,2})=Con(\tau^S_{4,2})$ & 0  & 0 & 10 & 30 & 42 & 30 & 10 & 0 & 0 \\
$Ch(\tau^S_{4,3})=Con(\tau^S_{4,3})$ & -4  & -12 & -16 & -8 & 0 & 0 & 0 & 0 & 0\\
$Con(\tau^S_{4,1})$ & 0  & 0 & 0 & 0 & 0 & 8 & 16 & 12 & 4 \\
\end{tabular} \\

In fact this gives another example that $Eu_{\tau^S_{4,2}}(\tau^{S\circ}_{4,3})=1$, but $\tau^S_{4,2}$ is singular at $\tau^S_{4,3}$. 

\begin{obse}
All the sequences appeared above are log concave.
\end{obse}

\bibliographystyle{plain}
\bibliography{ref.bib}
\end{document}